\documentclass[12pt]{amsart}
\usepackage{amsmath,amssymb,amsbsy,amsfonts,amsthm,latexsym,
               amsopn,amstext,amsxtra,euscript,amscd}
\usepackage{color}
\newtheorem{theorem}{Theorem}[section]

\newtheorem{corollary}{Corollary}[section]
\newtheorem{lemma}{Lemma}[section]

\newtheorem{remark}{Remark}
\begin{document}

\title{A simultaneous approximation problem for exponentials and logarithms}
\author{Veekesh Kumar and Riccardo Tosi}

\begin{abstract} 
Let $\alpha_1,\alpha_2$ be non-zero algebraic numbers such that $\frac{\log \alpha_2}{\log\alpha_1}\notin\mathbb{Q}$ and let $\beta$ be a quadratic irrational number. In this article, we prove that the values of two relatively prime polynomials $P(x,y,z)$ and $Q(x,y,z)$ with integer coefficients are not too small at the point $\left(\frac{\log\alpha_2}{\log \alpha_1},\alpha_1^\beta, \alpha_2^\beta \right)$. We also establish a measure of algebraic independence of those numbers among $\frac{\log\alpha_2}{\log \alpha_1}$, $\alpha^\beta_1$ and $\alpha^\beta_2$ which are algebraically independent. 
\end{abstract}
\address[Veekesh Kumar]{Department of Mathematics, Indian Institute of Technology, Dharwad  580011, Karnataka, India.}

\email[]{veekeshk@iitdh.ac.in}
\address[Riccardo Tosi]{ University of Duisburg-Essen, 45127 Essen, Germany.}

\email[]{riccardo.tosi@uni-due.de}
\pagenumbering{arabic}


\maketitle

\section{Introduction}
In 1950, A. O. Gelfond and N. I. Feldman \cite{gelfond1} established their well known measure of algebraic independence of the numbers $\alpha^\beta$ and $\alpha^{\beta^2}$,  where $\alpha \neq 0, 1$ and $\beta$ is a cubic irrational. More precisely, the following:  
\vspace{.1cm}

{\it Let $\varepsilon$ be a positive real number and let $P(x,y)\in\mathbb{Z}[x,y]$ be a non-zero polynomial. Then there exists an effectively computable constant $C(\varepsilon)>0$ satisfying  $t(P):=\deg P+\log H(P)>C(\varepsilon)$ and 
$$
\log|P(\alpha^\beta, \alpha^{\beta^2})|>-\exp(t(P)^{4+\varepsilon}).
$$}
The numbers $\alpha^\beta$ and $\alpha^{\beta^2}$ had been proven algebraically independent  by A. O. Gelfond \cite{gelfond}.  In 1968, A. A. Shmelev \cite{aa} (see also R. Tijdeman \cite[Consequence 3.2]{tijdeman1} and M. Waldschmidt \cite[Page 592, Corollary 15.29]{waldschmidt1}) proved that the transcendence degree of the field generated by the numbers $\frac{\log\alpha_2}{\log\alpha_1}$, $\alpha_1^\beta$, and $\alpha_2^\beta$ is at least $2$ if $\alpha_1$, $\alpha_2$ are multiplicatively independent algebraic numbers over $\mathbb{Q}$ and $\beta$ is quadratic irrational. In 1979,  W. D. Brownawell \cite{bro1} improved  the above  well-known measure of algebraic independence  of A. O. Gelfond and N. I. Feldman for the numbers  $\alpha^\beta$ and $\alpha^{\beta^2}.$  The elliptic analogue of the later result was established by R. Tubbs \cite{tubbs}. The main purpose of this article is to provide analogous results to those of W. D. Brownawell, as proved in \cite{bro1}, for the numbers $\frac{\log\alpha_2}{\log\alpha_1}$, $\alpha_1^\beta$, and $\alpha_2^\beta$. Here is our first theorem.
\begin{theorem}
\label{theorem: main}
Let $\alpha_1, \alpha_2$  be multiplicatively independent algebraic numbers, and let $\beta$ be a quadratic irrational. Then for all $C_1>0$ there exists an effectively computable constant $C>0$, depending only on $C_1, \beta$, $\alpha_1$, $\alpha_2$ and $\frac{\log \alpha_2}{\log\alpha_1}$, which satisfies the following property. Suppose that there are relatively prime polynomials $P(x, y)$, $Q(x, y, z)\in\mathbb{Z}[x, y, z]$ satisfying 
\begin{equation}\label{eq1}
\log \max\left\{\left|P\left(\frac{\log \alpha_2}{\log \alpha_1}, \alpha_1^\beta\right)\right|,\left|Q\left(\frac{\log \alpha_2}{\log \alpha_1}, \alpha_1^\beta, \alpha^\beta_2\right)\right| \right\}\leq -r^C.
\end{equation}
Here, we set 
\begin{align*}
\log r =& A^2 B \deg_y P \deg_z Q, \\
A=&\deg_yP (\deg_x Q+ \deg_z Q) + \deg_x P (\deg_y Q+ \deg_z Q),\\
     B=&(\deg_yQ+\deg_zQ)(\deg_yP+\deg_xP)\\
     & +\deg_yP(\deg_xQ+\deg_z Q+\log H( Q) + \deg Q)\\
     &+\log H( P)(\deg_y Q+\deg_z Q).
\end{align*}
Then there exists a polynomial $U(x)\in\mathbb{Z}[x]$ such that 
\begin{equation}\label{eq2}
\log \left|U\left(\frac{\log \alpha_2}{\log \alpha_1}\right)\right|<-r^{C/4}
\end{equation}
and 
$$
\deg U+\log H(U)<r^{C/ C_1}.
$$
\end{theorem}
\begin{remark}
    Under the same assumptions of Theorem~\ref{theorem: main}, one can also produce a polynomial $V\in\mathbb{Z}[x]$ such that 
    \begin{equation}
\log \left|V\left(\alpha_1^\beta\right)\right|<-r^{C/4}
\end{equation}
and 
$$
\deg V+\log H(V)<r^{C/ C_1}.
$$
This follows from the proof of Theorem~\ref{theorem: main} just by interchanging the roles of $\frac{\log \alpha_2}{\log \alpha_1}$ and $\alpha_1^\beta$.
\end{remark}
This theorem provides an analogue to Theorem 3 in \cite{bro1} for the numbers $\frac{\log\alpha_2}{\log\alpha_1}$, $\alpha_1^\beta$, and $\alpha_2^\beta$. As an application of Theorem \ref{theorem: main}, we deduce the following result, which says that the values of two relatively prime polynomials are not too small.
\begin{theorem}\label{theorem:main1}
  Let $\alpha_1$, $\alpha_2$ and $\beta$ be as in Theorem~\ref{theorem: main}. Then there is an effectively computable constant $C_2$ such that for all coprime polynomials $P(x,y)$, $Q(x,y,z)\in\mathbb{Z}[x,y,z]$, we have   
  \begin{equation}\label{eq3}
\log\max\left\{\left|P\left(\frac{\log \alpha_2}{\log \alpha_1}, \alpha_1^\beta\right)\right|,\left|Q\left(\frac{\log \alpha_2}{\log \alpha_1}, \alpha_1^\beta, \alpha^\beta_2\right)\right| \right\}> -r^{C_2},
\end{equation}
  where $r$ is defined as in Theorem~\ref{theorem: main}.
\end{theorem}
We have the following important remarks regarding Theorem \ref{theorem:main1}.
\begin{remark}\label{rmk3}
It is possible to deduce a non-trivial lower bound for arbitrary relatively co-prime polynomials $P(x,y,z), Q(x,y,z)$ with integer coefficients. This can be seen as follows: If $z$ occurs in both $P$ and $Q$, then we form the resultant $R$ to eliminate $z$ between the smallest prime factors $P_1$, $Q_1$ of $P$ and $Q$, respectively. Thus, we obtain a non-zero polynomial $R\in\mathbb{Z}[x,y]$. We then apply Theorem \ref{theorem:main1} to the polynomials $R$ and $Q_1$ to get the conclusion.
\end{remark}
\begin{remark}\label{rmk1}
In \eqref{eq3}, any two of the three numbers $\frac{\log\alpha_2}{\log\alpha_1}$, $\alpha^\beta_1$ and $\alpha^\beta_2$ could appear in the polynomial $P(x,y)$.
\end{remark}
\begin{remark}\label{rmk2}
If only two of the numbers $\frac{\log\alpha_2}{\log\alpha_1}$, $\alpha^\beta_1$ and $\alpha^\beta_2$ are involved in both polynomials $P$ and $Q$, we can eliminate any common occurrence by forming the resultant and then use Theorem~\ref{waldschmidt} from the last section of this paper to obtain the conclusion of Theorem \ref{theorem:main1}.
\end{remark}
Suppose that the numbers $\frac{\log\alpha_2}{\log\alpha_1}$, $\alpha^\beta_1$ and $\alpha^\beta_2$  are algebraically dependent over $\mathbb{Q}$. Then we have $Q\left(\frac{\log \alpha_2}{\log \alpha_1}, \alpha_1^\beta, \alpha^\beta_2\right)=0$ for some non-zero polynomial $Q\in\mathbb{Z}[x,y,z].$ On the other hand, at least two of them are algebraically independent, say $\theta_1,\theta_2\in\left\{\frac{\log\alpha_2}{\log\alpha_1}, \alpha^\beta_1, \alpha^\beta_2\right\}$. By Remark \ref{rmk1}, we can interchange the roles of three numbers in Theorems \ref{theorem:main1} and \ref{theorem: main}. Thus, we deduce the following measure of  algebraic independence.
\begin{corollary}
Let $\theta_1$ and $\theta_2$ be as above. Then there is an effectively computable constant $C>0$ such that for every non-zero polynomial $P(x,y)\in\mathbb{Z}[x,y]$, we have 
$$
\log|P(\theta_1,\theta_2)|>\exp(-C\deg_yP(\deg_y P+\deg_x P)^2(2\deg_yP+\deg_xP+\log H(P))).
$$
\end{corollary}
Our paper is organized as follows. In Section 2, we collect some lemmas useful for our proofs, and in Section 3 we construct an auxiliary polynomial in $\frac{\log\alpha_2}{\log\alpha_1}, \alpha^\beta_1$ and $\alpha^\beta_2$, which are required to proof our Theorem \ref{theorem:main1}. In Section 4, we complete the proof of Theorems \ref{theorem: main} and \ref{theorem:main1}. The idea of the proof of our results is similar to that of Brownawell \cite{bro1}, with several modifications.

\section{Preliminaries}
Let $\omega_1,\omega_2, \omega_3$ be transcendental numbers and let $F=\mathbb{Q}(\omega_1,\omega_2, \omega_3)$. Let $K$ be a finite extension of $\mathbb{Q}$ and let $\gamma_1,\dots, \gamma_n$ be algebraic integers such that $K=\mathbb{Q}(\gamma_1,\dots, \gamma_n)$. For a number $\omega\in \mathbb{Z}[\gamma_1,\dots, \gamma_n,\omega_1,\omega_2,\omega_3]$, we will speak of the degree and height of $\omega$ to refer to the degree and height over $\mathbb{Z}$ of a given polynomial $P\in \mathbb{Z}[x_1,\dots, x_n,y_1,y_2,y_3]$ such that $\omega=P(\gamma_1,\dots, \gamma_n,\omega_1,\omega_2,\omega_3)$. This notion depends, of course, on the choice of a polynomial expression for $\omega$, but this choice will always be clear from the context. We remark that the height $H(P)$ of a polynomial $P\in \mathbb{Z}[x_1,\dots, x_n,y_1,y_2,y_3]$ is defined as the maximum modulus of its coefficients.
\smallskip

We need the following general version of Siegel's lemma. 
\begin{lemma}[Siegel's Lemma \cite{bro2}]\label{siegel:lemma}  
Let $M$ and $N$ be integers such that $N\geq 16 M$. Let  $u_{i,j}\in\mathbb{Z}[y_1,\ldots,y_6]$ satisfying 
$$
\deg u_{i,j}\leq d, \quad H(u_{i,j})\leq A,
$$
where $A\geq 1$. Then the system of equations 
\begin{equation}
\sum_{j=1}^Nu_{ij}x_j=0, \quad (1\leq i\leq M)
\end{equation}
has  a non-trivial solution in $\mathbb{Z}[y_1,\ldots,y_6]$ satisfying 
$$
\deg x_i\leq 3 d,\quad H(x_i)\leq (1+d)^6 A N.
$$ 
\end{lemma}
The following lemma provides an estimate for the height of a polynomial. 
\begin{lemma}
    \label{lemma: height for norm}
    Consider a polynomial $P\in \mathbb{C}[x_1,\dots,x_r,y_1,\dots,y_s]$ and $\alpha_1,\dots,\alpha_s\in \mathbb{C}$. Set $Q=P(x_1,\dots,x_k,\alpha_1,\dots,\alpha_s)$. Then
    \[
    H(Q) \le \deg_y P\,  H(P) |\alpha_1\dots \alpha_s|^{\deg_y P}. 
    \]
\end{lemma}
\begin{proof}
    Let us write
    \[
    P=\sum_{i,j} a_{i,j} x_1^{i_1}\dots x_r^{i_r} y_1^{j_1}\dots y_s^{j_s} 
    \]
    with $a_{i,j}\in\mathbb{C}$. The coefficients of $P(x_1,\dots,x_r,\alpha_1,\dots,\alpha_s)$ satisfy
    \[
    \left| \sum_j a_{i,j} \alpha_1^{j_1}\dots \alpha_s^{j_s} \right| \le \deg_y P\,  H(P) |\alpha_1\dots \alpha_s|^{\deg_y P}.
    \]

\end{proof}
\begin{lemma}[Schwarz Lemma \cite{lang}]\label{thm:schwarz}
		Let $f$ be a holomorphic function on an open subset of $\mathbb{C}$ containing the ball $B_1$ of radius $\rho_1>0$ centred at the origin. Suppose that $f$ has $n$ zeros, counted with multiplicity, in the ball $B_2$ of radius $\rho_2<\frac{1}{3}\rho_1$ centred at the origin. Then
		\[
		|f|_{\rho_2}\le \left(\frac{3\rho_2}{\rho_1}\right)^n |f|_{\rho_1}.
		\]
	\end{lemma}
	\begin{proof}
		Let $a_1,\dots, a_n$ be zeros of $f$ in $B_2$. Then the function
		\[
		g(z)=f(z)\prod_{i=1}^n \frac{1}{z-a_i}
		\]
		is holomorphic on $B_1$, so we may apply the maximum modulus principle to deduce that 
		\[
		|g|_{\rho_2}\le |g|_{\rho_1}= |f|_{\rho_1} \max_{|z|=\rho_1}\prod_{i=1}^n \left|\frac{1}{z-a_i}\right|\le |f|_{\rho_1}\left(\frac{1}{\rho_1-\rho_2}\right)^n.
		\]
		On the other hand, we have
		\[
		|g|_{\rho_2}=|f|_{\rho_2}\max_{|z|=\rho_2} \prod_{i=1}^n \left|\frac{1}{z-a_i}\right|\ge \frac{|f|_{\rho_2}}{(2\rho_2)^n}.
		\]
		Since by assumption $\rho_1-\rho_2>\rho_1-\frac{1}{3}\rho_1=\frac{2}{3}\rho_1$, this chain of inequalities finally yields
		\[
		|f|_{\rho_2}\le \left(\frac{2\rho_2}{\rho_1-\rho_2} \right)^n|f|_{\rho_1}\le \left(\frac{3\rho_2}{\rho_1}\right)^n |f|_{\rho_1}.
		\]
	\end{proof}
We require the following special case of Tijdeman's result proved in \cite[Theorem 3]{tijdeman}.
\begin{lemma}\label{lem:tijdeman}
Let $R, D_1, D_2,k$ be positive integers. Consider the function 
$$
F(z)=\sum_{\lambda_1=0}^{D_1}\sum_{\lambda_2,\lambda_3=0}^{D_2}\psi(\lambda) z^{\lambda_1} e^{\gamma_{\lambda_2,\lambda_3} z},
$$
where $\psi(\lambda)\in \mathbb{C}$ and $\gamma_{\lambda_2,\lambda_3}=\lambda_2 \log \alpha_1+\lambda_3 \beta\log \alpha_1$, and set 
\begin{align*}
a&=\max_{0\leq \lambda_2,\lambda_3\leq D_2}\{1, |\gamma_{\lambda_2,\lambda_3}|\}\leq c_1 D_2, \quad b=\max_{\sigma=0,\ldots,s-1}\{1, |\beta_\sigma|\}\leq c_2 k R\\
m&=\displaystyle\sum_{\lambda_2, \lambda_3} m_{\lambda_2,\lambda_3}=(D_2+1)^2(D_1+1).
\end{align*}
Moreover, we let $A=\displaystyle\max_{\lambda}|\psi(\lambda)|$ and we call the numbers $z_\mu$ (with $\mu=(\mu_0, \ldots,\mu_3)$, $0\leq |\mu_i|\leq k R$ for some integer $k>0$) simply $\beta_0,\ldots,\beta_{s-1}$, and further set 
$$
a_0=\displaystyle\min_{(\lambda_2,\lambda_3)\neq (\lambda'_2, \lambda'_3)}(1, |\gamma_{\lambda_2,\lambda_3}-\gamma_{\lambda'_2,\lambda'_3}|), 
$$
and 
$$
E=\displaystyle\max_{|\mu_i|\leq kR}|F(z_\mu)|.
$$
Then, if $(k R+1)^4\geq 2 (D_1+1)(D_2+1)^2$, we have 
$$
A\leq (k R+1)^4\left(\frac{6 c_1 D_1}{a_0(D_2+1)}\cdot\frac{(D_1+1)(D_2+1)^2}{b}\right)^{(D_1+1)(D_2+1)^2}\left(\frac{72 c_2 kR}{b_0(kR+1)^2}\right)^{(kR+1)^4} E.
$$
\end{lemma}
The proof of Theorem \ref{theorem: main} essentially depends on an elimination process, which is done by forming resultants or semi-resultants. The following lemma on semi-resultant plays a key role in the proof of Theorem \ref{theorem: main}.
    \begin{lemma}{\cite{bro3}}
        \label{lemma: chudnovsky}
        Let $P(x)=p_0x^m + \dots + p_m$ and $Q(x)=q_0x^n+\dots + q_n$ be polynomials with complex coefficients.
        \begin{enumerate}
            \item The semi-resultant $r$ of $P$ and $Q$ can be written as a polynomial over $\mathbb{Z}$ in the $p_i$ and $q_i$ with 
            \[
            \deg_{p_i} r \le n, \qquad \deg_{q_i} r\le m, \qquad \text{H}(r ) \le 8^{mn}.
            \]
            \item Let $\theta\in \mathbb{C}$ and let $P_1(x),P_2(x)\in \mathbb{C}[x]$ be relatively prime monic polynomials such that 
            \[
            \max \{ |P_1(\theta)|, \, |Q_1(\theta)|\} < 1.
            \]
            Suppose that $P_1$ divides $P$ and $Q_1$ divides $Q$. Then there is an absolute constant $k$ such that
            \[
            |r|\le \max \{ |P_1(\theta)|, \, |Q_1(\theta)|\} \max \{ 1, \, |\theta|\}^{\deg P_1\, \deg Q_1} \text{H}(P)^n\text{H}(Q)^m e^{mnk}.
            \]
        \end{enumerate}
    \end{lemma}
We need the following lemma on semi-resultants for the polynomials in several variables.
    \begin{lemma}
        \label{lemma: degree and height of semi-resultant}
        Let $P,Q\in \mathbb{C}[x_1,\dots, x_k, y]$ and let $r\in \mathbb{C}[x_1,\dots,x_k]$ be the semi-resultant of $P$ and $Q$ over $\mathbb{C}[x_1,\dots, x_k]$. Write $m=\deg_y P$ and $n=\deg_y Q$. Then 
        \begin{gather*}
        \deg_{x_i} r \le n \deg_{x_i} P + m \deg_{x_i} Q; \\
        \text{H}(r)\le 81^{mn} 2^{k(n \deg_{x} P + m \deg_{x} Q)} \text{H} (P)^n \text{H}(Q)^m. 
        \end{gather*}
    \end{lemma}
    \begin{proof}
        Write $P(y)=p_0y^m + \dots + p_m$ and $Q(y)=q_0y^n+\dots + q_n$. Then we write uniquely
        \[
        r = \sum_{I} r_I \, p_0^{a_0}\dots p_m^{a_m} q_0^{b_0}\dots q_n^{b_n},
        \]
        where $I=(a_0,\dots, a_m, b_0, \dots, b_n)$ with $\sum_{j=0}^m a_j \le \deg_p r$ and  $\sum_{j=0}^n a_j \le \deg_q r$. Since $\deg_{x_i} P = \max_j \deg_{x_i} p_j$ and $\deg_{x_i} Q = \max_j \deg_{x_i} q_j$, by Lemma~\ref{lemma: chudnovsky} we deduce that 
        \begin{align*}
        \deg_{x_i} (p_0^{a_0}\dots p_m^{a_m} q_0^{b_0}\dots q_n^{b_n}) & \le \left( \sum_{j=0}^m a_j \right) \max_j \deg_{x_i} p_j + \left( \sum_{j=0}^n b_j \right) \max_j \deg_{x_i} q_j \\ 
        &\le n\deg_{x_i} P + m\deg_{x_i} Q,
        \end{align*}
        whence the claim on the degree of $r$ follows at once.
        \smallskip
        
        Now, let us estimate the height of $r$. From Lemma~\ref{lemma: chudnovsky}, we have $|r_I|\le 8^{mn}$. Hence,  by  Gelfond's Lemma \cite[Page 47, Lemma 3.3]{nesterenko}, we obtain that 
        \[
        \text{H}(p_0^{a_0}\dots p_m^{a_m} q_0^{b_0}\dots q_n^{b_n}) \le 2^d \prod_{i=0}^m \text{H}(p_i)^{a_i} \prod_{i=0}^n \text{H}(q_i)^{b_i},
        \]
        where $d=\sum_{s=1}^k \left( \sum_{i=0}^m a_i\deg_{x_s} p_i + \sum_{i=0}^n b_i\deg_{x_s} q_i \right)$. The similar  estimate as above shows that $d\le k(n\deg_x P+ m \deg_x Q)$. Thus, 
        \[
        \text{H}(p_0^{a_0}\dots p_m^{a_m} q_0^{b_0}\dots q_n^{b_n}) \le 2^{k(n\deg_x P+ m \deg_x Q)} \text{H}(P)^n \text{H}(Q)^m.
        \]
        It only remains to estimate the number of possible indices $I$. These are at most $4m^{n+1}n^{m+1}$, hence the desired inequality follows by observing that $8^{mn}\cdot4m^{n+1}n^{m+1}\le 81^{mn} $.
    \end{proof}
\begin{lemma}\label{irreducible}{\cite{bro1}}
Suppose $w\in\mathbb{C}$, and $P(x)$ is a non-zero polynomial with integer coefficients such that $|P(w)|<-\exp(\lambda d(h+d))$, where $\lambda>3$, $d\geq \deg P$, $e^h\geq H(P)$. Then there is a factor $Q(x)$ of $P(x)$ such that $Q(x)=u(x)^v$, for some irreducible $u(x)\in\mathbb{Z}[x]$ is irreducible and $v\in\mathbb{N}$, and satisfying 
$$
\log |Q(w)|<-(\lambda-1)d(h+d).
$$
\end{lemma}

\section{Construction of auxiliary polynomial}\label{auxiliary}
In this section, we construct an auxiliary polynomial involving $\alpha_1, \alpha_2,\beta,\frac{\log\alpha_2}{\log\alpha_1},\alpha^\beta_1, \alpha^\beta_2$, which  requires for the proof of Theorem \ref{theorem: main}.
\smallskip

\noindent{\it Step 1.} (Construction of an auxiliary function). Since $\alpha_1,\alpha_2$ are multiplicatively independent algebraic numbers and $\beta$ quadratic irrational, by Gelfond-Schneider theorem (Schneider-Lang theorem) the numbers $\frac{\log \alpha_2}{\log \alpha_1}, \alpha^\beta_1$ and $\alpha^\beta_2$ are transcendental. Let $K$  be the Galois closure of $\mathbb{Q}(\alpha_1,\alpha_2,\beta)$ with $d_1=[K:\mathbb{Q}]$ and $L=K\left(\frac{\log \alpha_2}{\log \alpha_1}, \alpha^\beta_1, \alpha^\beta_2\right)$. By the primitive element theorem there exists an algebraic integer $\omega$ over $\mathbb{Q}$ and $\omega_1$ over $K\left(\frac{\log \alpha_2}{\log \alpha_1}, \alpha^\beta_1, \alpha^\beta_2\right)$ such that  $K$ is written as $\mathbb{Q}(\omega)$ and $L=\mathbb{Q}\left(\omega, \omega_1,\frac{\log \alpha_2}{\log \alpha_1}, \alpha^\beta_1, \alpha^\beta_2\right)$.
Let $S\in\mathbb{N}$ such that $S\alpha_1, S\alpha_2, S\beta\in \mathbb{Z}[\omega]$ and $T\in\mathbb{Z}\left[\omega,\frac{\log \alpha_2}{\log\alpha_1},\alpha^\beta_1,\alpha^\beta_2\right]$ such that 
$$
T\alpha^\beta_1, T\alpha^\beta_2,\frac{\log \alpha_2}{\log \alpha_1}\in\mathbb{Z}\left[\omega,\omega_1, \frac{\log \alpha_2}{\log\alpha_1},\alpha^\beta_1,\alpha^\beta_2\right].
$$
$ $\\

Let $d$ be the smallest positive integer such that $d\beta$ is an algebraic integer and write $(d\beta)^2=b_0+b_1 (d \beta)$ for $b_0, b_1\in\mathbb{Z}$. 
We fix a constant $C_1\ge 1$ and we choose $C>C_1$ large enough. Recall that $r$ is defined as
 \[
    \log r= A^2 B \deg_y P \deg_z Q,
 \]
 with $A$ and $B$ as in the statement of Theorem~\ref{theorem: main}.
 Let us also define
 \[
    N_0=\left[ r^{\frac{C}{11C_1}} \right], \qquad  N_1=\left[ r^{\frac{2C}{11}} \right].
 \]
 For all $N$ such that $N_0\le N\le N_1$, we set 
 \[
D_1=\left[ \frac{N^2}{2^6(KA\deg_yP\deg_z Q+1)^2} \right], \quad D_2= \left[ 2^7(KA\deg_yP\deg_z Q+1)N\right],
 \]
 where $K$ is a large absolute constant. We wish to construct an auxiliary function
\begin{equation}
F(z)=\displaystyle\sum_{\lambda_1=0}^{D_1}\sum_{\lambda_2=0}^{D_2}\sum_{\lambda_3=0}^{D_2} \psi(\lambda_1,\lambda_2,\lambda_3)z^{\lambda_1} (e^{\log \alpha_1 z})^{\lambda_2}(e^{\beta\log \alpha_1 z})^{\lambda_3}
\end{equation}
where $\psi(\lambda_1,\lambda_2,\lambda_3)$ are numbers in $\mathbb{Q}\left(\beta, \alpha_1,\alpha_2,\frac{\log \alpha_2}{\log\alpha_1},\alpha^{\beta}_1, \alpha^\beta_2\right)$  to be determined in  such a way that $F$ satisfies the following property: for all $\mu=(\mu_0,\mu_1,\mu_2,\mu_3)\in \mathbb{Z}^4_{\geq 0}$ with $0\leq \mu_i\leq N$ for $0\leq i\leq 4$, we have 
$$
F({z_\mu})=0\quad\mbox{where}\quad z_\mu=\mu_0+\mu_1 d^2 \beta+\mu_2 \frac{\log \alpha_2}{\log \alpha_1}+\mu_3\frac{\log \alpha_2}{\log \alpha_1}d^2\beta.
$$
We apply Lemma \ref{siegel:lemma} (Siegel's lemma) to construct such a function $F$. The condition $F(z_\mu)=0$ for all $\mu$ with $0\leq \mu_i\leq N$ can be interpreted as a homogeneous linear system whose coefficients are given by the function  $z^{\lambda_1} \alpha^{\lambda_2 z}_1 \alpha^{\beta \lambda_3 z}_1$ evaluated at the point $z_\mu$. Therefore, this linear system has coefficients in $\mathbb{Q}\left(\beta, \frac{\log \alpha_2}{\log \alpha_1}, \alpha_1, \alpha_2, \alpha_1^\beta, \alpha_2^\beta\right)$.  

Let us define $a_{\lambda,\mu}=z^{\lambda_1}_\mu \alpha_1^{\lambda_2 z_\mu+\lambda_3 \beta z_\mu}$. Now we want to solve the  system of homogeneous equations $\displaystyle\sum_\lambda \psi(\lambda) a_{\lambda\mu}=0$ with coefficients $a_{\lambda \mu}\in \mathbb{Q}(\beta,\alpha_1,\alpha_2)\left(\frac{\log \alpha_2}{\log \alpha_1},\alpha^\beta_1,\alpha^\beta_2\right)$ and unknowns $\psi(\lambda):=\psi(\lambda_1,\lambda_2,\lambda_3)$. \\
\smallskip

By substituting the value of $z_\mu$, we get 
\begin{align*}
a_{\lambda,\mu}&=\left(\mu_0+\mu_1 d^2 \beta+\mu_2\frac{\log \alpha_2}{\log \alpha_1}+\mu_3
\frac
{\log \alpha_2}{\log \alpha}d^2 \beta\right)^{\lambda_1}\alpha^{\lambda_2\mu_0}_1 (\alpha_1^\beta)^{d^2 \lambda_2\mu_1} \alpha^{\lambda_3 \mu_2}_2(\alpha^\beta_2)^{d^2 \lambda_3\mu_3}\\
&(\alpha_1^\beta)^{\lambda_2 \mu_0}(\alpha_1^{d^2 \beta^2})^{\lambda_2 \mu_1}(\alpha_2^\beta)^{\lambda_3 \mu_2}(\alpha^{d^2 \beta^2}_2)^{\lambda_3\mu_3}\\
&=\left(\mu_0+\mu_1 d^2 \beta+\mu_2\frac{\log \alpha_2}{\log \alpha_1}+\mu_3\frac{\log \alpha_2}{\log \alpha}d^2 \beta\right)^{\lambda_1}\alpha^{\lambda_2(\mu_0+b_0\mu_1)}_1(\alpha_1^\beta)^{\lambda_2(\mu_0+d^2 \mu_1+\mu_1 b_1 d)}\\
&\alpha^{\lambda_3(\mu_2+b_0\mu_3)}_2(\alpha_2^\beta)^{\lambda_3(\mu_2+d^2 \mu_3+\mu_3 b_1 d)}.
\end{align*}
We have used the fact that $(d\beta)^2=b_0+b_1 (d\beta)$. We multiply each row by a suitable power of $\alpha_1$ and $\alpha_2$ to remove the negative powers. More precisely, we multiply the equation $\displaystyle\sum_{\lambda,\mu}\psi(\lambda)a_{\lambda\mu}=0$ by 
\begin{align*}
C_\mu=&\alpha^{D_2\max\{0,-\mu_0-b_0\mu_1\}}_1(\alpha^\beta_1)^{D_2\max\{0,-\mu_0-d^2 \mu_1-\mu_1 b_1 d\}}\times \\
& \times\alpha^{D_2\max\{0,-\mu_2-b_0\mu_3\}}_2(\alpha^\beta_2)^{D_2\max\{0,-\mu_2-d^2 \mu_3-\mu_3 b_1 d\}}.
\end{align*}
In order to make the coefficients $a_{\lambda,\mu}\in \mathbb{Z}\left[\omega,\omega_1,\frac{\log\alpha_2}{\log \alpha_1},\alpha^\beta_1,\alpha^\beta_2\right]$, we multiply $S^{D_1}$  and $T^{D_2}$.
Overall, this factor $C_\mu$ has degree at most:
$$
\deg C_\mu\leq D_2 N(1+|b_0|+1+d^2+d|b_1|+1+|b_0|+1+d^2+d|b_1|)=D_2 N(4+2|b_0|+2 d |b_1|+2 d^2).
$$
Thus the polynomial $C_\mu S^{D_1} T^{D_2} a_{\lambda\mu}\in\mathbb{Z}\left[\omega,\omega_1,\frac{\log \alpha_2}{\log \alpha_1},\alpha^\beta_1,\alpha^\beta_2\right]$ has degree at most 
\[
    D_1+D_2 N(4+2|b_0|+2d |b_1|+2 d^2).
\]

\smallskip

Let us write for short $X_0=1$, $X_1=\beta$, $X_2=\frac{\log \alpha_2}{\log \alpha_1}$ and $X_3=\frac{\log\alpha_2}{\log \alpha_1}\beta$. Then $z_\mu$ equals
$$
(\mu_0 X_0+\mu_1 d X_1+\mu_2 X_2+\mu_3 d X_3)^{\lambda_1}=\displaystyle\sum_{p_0+p_1+p_2+p_3=\lambda_1}{\lambda_1 \choose p_0,p_1,p_2,p_3} d^{p_1+p_3} \prod_{i=0}^3 (\mu_iX_i)^{p_i}.
$$
Since ${\lambda_1\choose p_0,p_1,p_2,p_3}\leq 4^{\lambda_1}$ and $\mu_i\le N$, the height of $z_{\mu}^{\lambda_1}$ is at most $(4 d)^{\lambda_1}N^{\lambda_1}<(4 d)^{D_1}N^{D_1}$. Thus we have: 
\begin{align*}
\deg(S^{D_1} T^{D_2} C_\mu a_{\lambda,\mu}) &\leq D_1+D_2 N(4+2|b_0|+2 d|b_1|+2 d^2), \\
H(S^{D_1} T^{D_2} C_\mu a_{\lambda,\mu}) &\leq (4 dN)^{D_1}c^{D_2 N}
\end{align*}
for some absolute constant $c>0$.
\smallskip

The linear system $\displaystyle\sum_{\lambda}\psi(\lambda)C_\mu a_{\lambda,\mu}=0$ has 
\begin{align*}
 \mbox{the number of unknowns}&= (D_1+1)(D_2+1)^2\\
 \mbox{the number of equations}&=(N+1)^4.
\end{align*}
Since $2^6 (N+1)^4<(D_1+1)(D_2+1)^2$, by Lemma~\ref{siegel:lemma} (Siegel's lemma), the above system of equations has a non-trivial solution given by $\psi(\lambda)\in\mathbb{Z}\left[\omega,\omega_1,\frac{\log\alpha_2}{\log\alpha_1},\alpha^\beta_1,\alpha^\beta_2\right]$ which satisfies
\begin{align*}
 \deg \psi(\lambda)&\leq 6 D_1+6 D_2 N (4+2|b_0|+2 d|b_1|+2 d^2)\\
 H(\psi(\lambda))&\leq (1+D_1+D_2 N (4+2|b_0|+2 d|b_1|+2 d^2))^{12} \times \\
                & \times (D_1+1)(D_2+1)^2(4d N)^{D_1}c^{D_2 N}.
\end{align*}
Since $2^7(N+1)^4<(D_1+1)(D_2+1)^4$, we see that there are two effectively computable constants $c_1, c_2$ such that 
$\deg \psi(\lambda)\leq c_1(D_1+D_2 N)$ and $\log H(\psi(\lambda))\leq c_2 D_1\log N$, since $D_2 N<D_1\log N$ for $N$ big enough. To be more precise about these constants, if $C>\mbox\{|b_0|, |b_1|, d^2\}$, we may then choose $c_1=60 C$, $c_2=\log (4 C)$. Thus, we have a  function 
$$
F(z)=\displaystyle\sum_{\lambda_1=0}^{D_1}\sum_{\lambda_2=0}^{D_2}\sum_{\lambda_3=0}^{D_2} \psi(\lambda_1,\lambda_2,\lambda_3)z^{\lambda_1}\alpha^{\lambda_2 z}_1\alpha^{\beta \lambda_3 z}_1
$$
such that $F(z_\mu)=0$ for all $z_\mu=\mu_0+\mu_1 d^2 \beta+\mu_2 \frac{\log \alpha_2}{\log\alpha_1}d^2 \beta$ with $0\le \mu_i\leq N$, and there are polynomials $\Phi_\lambda\in\mathbb{Z}[x_1,\ldots,x_5]$ such that $\Phi_\lambda(\omega,\omega_1,\frac{\log\alpha_2}{\log \alpha_1},\alpha^\beta_1,\alpha^\beta_2)=\psi(\lambda)$, and not all $\psi(\lambda)$ are zero with $\deg \Phi_\lambda\leq c_1(D_1+D_2 N)$ and $\log H (\Phi_\lambda)\leq c_2 D_1\log N$.
\smallskip
For the moment, the coefficients $\psi(\lambda)$ of $F$ are polynomials in $\mathbb{Z}\left[\beta, \alpha_1,\alpha_2,\frac{\log \alpha_2}{\log\alpha_1},\alpha^{\beta}_1, \alpha^\beta_2\right]$. Since $\beta,\alpha_1,\alpha_2$ are algebraic, we can remove them by taking a norm in such a way that the new coefficients of $F$ lie in $\mathbb{Z}\left[\frac{\log \alpha_2}{\log\alpha_1},\alpha^{\beta}_1, \alpha^\beta_2\right]$. Let $\sigma_1,\dots, \sigma_\delta$ be the automorphisms of $K\left(\frac{\log \alpha_2}{\log\alpha_1},\alpha^{\beta}_1, \alpha^\beta_2\right)$ which fix $\mathbb{Q}\left(\frac{\log \alpha_2}{\log\alpha_1},\alpha^{\beta}_1, \alpha^\beta_2\right)$. Moreover, let $t>0$ be an integer such that $t\beta, t\alpha_1, t\alpha_2$ are algebraic integers.\\
For all $i=1,\dots, \delta$ define $F_i$ to be the function obtained from $F$ by replacing each $\psi(\lambda)$ by $\sigma_i(\psi(\lambda))$. The product $F'=t^{\delta C_1(D_1+D_2N)} F_1\dots F_\delta$ is a polynomial in the functions $z$, $\alpha_1^z$, $\alpha_1^{\beta z}$ which is fixed by $\sigma_1,\dots, \sigma_\delta$, so its coefficients $\psi'(\lambda)$ for $\lambda_1=0,\dots, \delta D$ and $\lambda_2,\lambda_3=0,\dots,\delta D_2$ belong to $\mathbb{Q}\left(\frac{\log \alpha_2}{\log\alpha_1},\alpha_1^{\beta}, \alpha_2^\beta\right)$. Since $\deg \psi(\lambda)\leq C_1(D_1+D_2 N)$, the factor $t^{\delta C_1(D_1+D_2 N)}$ provides enough powers of $t$ to ensure that each $\psi'(\lambda)$ actually lies in  $\mathbb{Z}\left[\frac{\log \alpha_2}{\log\alpha_1},\alpha_1^{\beta}, \alpha_2^\beta\right]$. Moreover, notice that 
\[
\psi'(\lambda)=t^{\delta C_1(D_1+D_2 N)} \sum \sigma_1(\psi(\lambda^{(1)})) \dots \sigma_\delta(\psi(\lambda^{(\delta)})),
\]
the sum being taken over all $\lambda^{(j)}=(\lambda^{(j)}_1,\lambda^{(j)}_2,\lambda^{(j)}_3)$ such that $\displaystyle\sum_{j=1}^\delta \lambda^{(j)}_i=\lambda_i$ for $i=1,2,3$. Lemma~\ref{lemma: height for norm} then implies that
\[
\deg \psi'(\lambda) \le \delta C_1 (D_1+D_2 N), \qquad \log H(\psi'(\lambda)) \le D_1\log N+D_2 N.
\]
From now on, we will replace $F$ by $F'$ and continue to denote it by $F$.
$ $

\noindent Step 2. {\it Upper bound for $|F(z_\mu)|$}. 
\smallskip
We wish to apply Schwarz's lemma to the function $F$, as the high number of zeroes of $F$ allows us to keep its modulus relatively low in a ball that is not too large and contains all the known zeros of $F$.
\smallskip

Let $C'>0$ be a positive constant such that the known zeros of $F$ are contained in the ball of radius $R_2=C' N\log N$ (say $C'=1+ d^2 |\beta|+|\frac{\log \alpha_2}{\log \alpha_1}|+|\frac{\log \alpha_2}{\log \alpha_1}d^2 \beta|$). Fix some $\varepsilon>0$ such that $0<\varepsilon\le \frac{1}{2}$ and set $R_1=C'N^{1+\varepsilon}$. We will apply Schwarz's lemma to the function $F$ considering the balls of radius $R_1$ and $R_2$ centered at the origin. Notice that we may assume $R_2<\frac{1}{3} R_1$ for $N$ sufficiently large.
We estimate $|F(z)|$ with respect to $|z|$ as follows: 
$$
|F(z)|\leq \sum_{\lambda}|\psi(\lambda)||z|^{\lambda_1}|\alpha_1^{\lambda_1 z}||\alpha^{\beta \lambda_3 z}_1|\leq \sum_{\lambda}|\psi(\lambda)||z|^{\lambda_1} e^{D_2 |z|(\mathrm{Re}(\log\alpha_1)+\mathrm{Re}(\log\alpha_1\cdot\beta))}. 
$$
Since $|\psi(\lambda)|\ll H(\psi(\lambda))\deg (\psi(\lambda))^6\ll e^{c_2 D_1\log N+D_2 N}(D_1+D_2 N)^6$, we get that 
$$
|F(z)|\ll (D_1+1)(D_2+1)^2 e^{C_2 D_1\log N+D_2 N}(D_1+D_2 N )^6|z|^{D_1} e^{D_2 \mathrm{Re}(\log \alpha_1)+\mathrm{Re}(\log \alpha\cdot\beta)|z|}. 
$$
Consequently, for all $N$ sufficiently large, we have 
$$
\log |F(z)|\ll D_1\log N +D_2 N+ D_1 \log|z|+D_2 |z|,
$$
where the implied constant is independent of $N$. If we now take $|z|=R_1=C' N^{1+\varepsilon}$, we get 
$$
\log |F(z)|\ll (1+\varepsilon) D_1\log (C' N) + C' D_2 N^{1+\varepsilon}.
$$
On the other hand, in the estimate for $|F(z)|$ for $|z|=R_2$ provided by Schwarz's lemma, the factor determined by the number of zeroes is
\[
\left(\frac{3 R_2}{R_1}\right)^{(N+1)^4} = \left(\frac{3 \log N}{N^\varepsilon}\right)^{(N+1)^4} < \left(\frac{3 N^{\frac{\varepsilon}{2}}}{N^\varepsilon}\right)^{(N+1)^4} < 
\exp\left( -\frac{\varepsilon}{2} (N+1)^4 \log N\right). 
\]
By Schwarz's lemma, for $|z|\le R_1=C'N^{1+\varepsilon}$, we have
\[
\log |F(z)|\le C''((1+\varepsilon) D_1\log (C' N) + C' D_2 N^{1+\varepsilon})-\frac{\varepsilon}{2} (N+1)^4 \log N
\]
for some constant $C''>0$. 
Since $D_1<N^4$ and $D_2< N^{3-\varepsilon}$, we get that 
\[
\log |F(z)|\le -\frac{\varepsilon}{4} (N+1)^4 \log N
\]
whenever $|z|\le C' N\log N$.

\section{Proof of Theorem \ref{theorem: main}}

In this section we complete the proof of Theorem \ref{theorem: main} in four steps, which are analogous to those in \cite{bro1}. 
\smallskip

\textit{Step} $0$.
Let us suppose that the statement of Theorem \ref{theorem: main} does not hold for some sufficiently large constant $C>0$.
We are assuming that there are two coprime polynomials $P(x,y), Q(x,y,z)\in\mathbb{Z}[x,y,z]$ such that 
\[
\log\mbox{max}\left\{\left|P\left(\frac{\log \alpha_2}{\log \alpha_1}, \alpha_1^\beta\right)\right|,\left|Q\left(\frac{\log \alpha_2}{\log \alpha_1}, \alpha_1^\beta, \alpha^\beta_2\right)\right| \right\}\leq -r^C.
\]
This inequality together with the negation of the statement implies that $\deg_y P >0$ and that the leading coefficient $p(\frac{\log\alpha_2}{\log \alpha_1})$ of $P(\frac{\log\alpha_2}{\log \alpha_1}, y)$ satisfies 
\[
\log \left|p\left(\frac{\log\alpha_2}{\log \alpha_1}\right)\right| \ge -r^{C/4}.
\]
By Lemma 2 in \cite{bro1}, we may assume that $P$ and $Q$ are irreducible. By assumption, we have
\[
e^{-r^{C}}\geq \left|P\left(\frac{\log \alpha_2}{\log \alpha_1}, \alpha_1^\beta\right)\right| = \left|p\left(\frac{\log\alpha_2}{\log \alpha_1}\right)\right| \prod_{\text{$\xi$ root of $P(\frac{\log\alpha_2}{\log \alpha_1}, y)$}} \left|\alpha_1^\beta - \xi\right|.
\]
Since $P(\frac{\log\alpha_2}{\log \alpha_1}, y)$ has $\deg_y P$ roots, the closest root $\xi_1$ of $P(\frac{\log\alpha_2}{\log \alpha_1}, y)$ to $\alpha_1^\beta$ satisfies
\[
\log |\alpha_1^\beta- \xi_1|  < -\frac{r^{\frac{3}{4}C}}{\deg_y P}.
\]
Since $\deg_y P\le r$, we infer that 
\[
\log |\alpha_1^\beta- \xi_1|  < -r^{\frac{3}{4}C-1} \le -r^{\frac{7}{8}C}.
\]
Now let us write $Q(x,y,z)=\displaystyle\sum_{i,j,k\ge 0} a_{ijk} x^i y^j z^k$. Then
\[
Q\left(\frac{\log \alpha_2}{\log \alpha_1}, \alpha_1^\beta, z\right) - Q\left(\frac{\log \alpha_2}{\log \alpha_1}, \xi_1, z\right) = 
\sum_{i,j\ge 0} \sum_{k\ge 1} a_{ijk} \left(\frac{\log \alpha_2}{\log \alpha_1}\right)^i z^k (\alpha_1^\beta -\xi_1) \sum_{p=0}^{j-1}\alpha_1^{p\beta}\xi_1^{j-1-p},
\]
which shows that $\alpha_1^\beta-\xi_1$ formally divides $Q\left(\frac{\log \alpha_2}{\log \alpha_1}, \alpha_1^\beta, z\right) - Q\left(\frac{\log \alpha_2}{\log \alpha_1}, \xi_1, z\right)$. As a result:
\begin{align*}
\left| Q\left(\frac{\log \alpha_2}{\log \alpha_1}, \xi_1, \alpha_2^\beta \right) \right| & \le \left| Q\left(\frac{\log \alpha_2}{\log \alpha_1}, \alpha_1^\beta, \alpha_2^\beta \right) \right| + \left| Q\left(\frac{\log \alpha_2}{\log \alpha_1}, \alpha_1^\beta, \alpha_2^\beta \right) - Q\left(\frac{\log \alpha_2}{\log \alpha_1}, \xi_1, \alpha_2^\beta \right) \right| \\ 
& \le e^{-r^C} + |\alpha_1^\beta  -\xi_1| \left| \sum_{i,j\ge 0} \sum_{k\ge 1} a_{ijk} \left(\frac{\log \alpha_2}{\log \alpha_1}\right)^i \alpha_2^{k\beta} \sum_{p=0}^{j-1}\alpha_1^{p\beta}\xi_1^{j-1-p} \right|.
\end{align*}
By taking $C$ greater than $|\alpha_1^\beta|+1$ and $|\alpha_2^\beta|$ and by noticing that $|\xi_1|\le |\alpha_1^\beta| + |\alpha_1^\beta-\xi_1|\le |\alpha_2^\beta| +1 $, we see that 
\begin{gather*}
\left| \sum_{\substack{i,j\ge 0 \\ k\ge 1}}  a_{ijk} \left(\frac{\log \alpha_2}{\log \alpha_1}\right)^i \alpha_2^{k\beta} \sum_{p=0}^{j-1}\alpha_1^{p\beta}\xi_1^{j-1-p} \right| \\\le 
\deg Q \, \text{H}(Q) \, \max \{|\alpha_1^\beta|,|\xi_1|,1\}^{2\deg_y Q} \, \max \{|\alpha_2^\beta|,1\}^{\deg_z Q} \\
 \le \deg Q \, \text{H}(Q) C^{2\deg_y Q + \deg_z Q}.
\end{gather*}
Overall, combining the previous estimates, we obtain
\begin{gather*}
\left| Q\left(\frac{\log \alpha_2}{\log \alpha_1}, \xi_1, \alpha_2^\beta \right) \right| \\\le e^{-r^C} + \exp\left(-r^{\frac{7C}{8}}+(2\deg_y Q+\deg_z Q)\log C + \log \deg Q + \text{H}(Q)\right).
\end{gather*}
For $C$ sufficiently large, we conclude that
\[
\left| Q\left(\frac{\log \alpha_2}{\log \alpha_1}, \xi_1, \alpha_2^\beta \right) \right| \le \exp(-r^C) + \exp\left(-\frac{2}{3}r^{\frac{7C}{8}}\right) \le \exp\left(-\frac{1}{2}r^\frac{7C}{8}\right). 
\]
The polynomial $Q\left(\frac{\log \alpha_2}{\log \alpha_1}, \xi_1, z \right)$ does not vanish identically since $P$ and $Q$ are relatively prime. 
If  $Q\left(\frac{\log \alpha_2}{\log \alpha_1}, \xi_1, z \right)$ is constant, say $Q\left(\frac{\log \alpha_2}{\log \alpha_1}, \xi_1, z \right)=Q'\left(\frac{\log \alpha_2}{\log \alpha_1}, \xi_1\right)$, then consider the resultant $R\left(\frac{\log \alpha_2}{\log \alpha_1}\right)$ between $P\left( \frac{\log \alpha_2}{\log \alpha_1}, y\right)$ and $Q'\left(\frac{\log \alpha_2}{\log \alpha_1}, y\right)$ with respect to the variable $y$. Notice that $P\left( \frac{\log \alpha_2}{\log \alpha_1}, \xi_1\right)=0$, while 
\[
\left| Q'\left(\frac{\log \alpha_2}{\log \alpha_1}, \xi_1 \right) \right| =\left| Q\left(\frac{\log \alpha_2}{\log \alpha_1}, \xi_1, \alpha_2^\beta \right) \right| \le \exp\left(-\frac{1}{2}r^\frac{7C}{8}\right).
\]
Since $r$ is greater than $\deg P \deg Q$, $\deg P\log \text{H}(Q)$ and $\deg Q\log \text{H}(P)$, it follows from Lemma~\ref{lemma: chudnovsky} that $R$ satisfies the requirements of the statement of Theorem~\ref{theorem: main}. \\
We may thus assume that $Q\left(\frac{\log \alpha_2}{\log \alpha_1}, \xi_1, z \right)$ is not constant. Hence, we have a non-trivial lower bound for its leading coefficient $q\left(\frac{\log \alpha_2}{\log \alpha_1}, \xi_1\right)$, exactly as we did for $p\left(\frac{\log \alpha_2}{\log \alpha_1}\right)$. The same argument as before therefore shows that the closest root $\xi_2$ of $Q\left(\frac{\log \alpha_2}{\log \alpha_1}, \xi_1, z \right)$ to $\alpha_2^\beta$ satisfies
$$
\log |\alpha_2^\beta - \xi_2|\le -r^{\frac{3C}{4}}. 
$$
$ $\\
\noindent{\it Step 1.}~{(Applying Tijdeman's lemma)}.
Since $2(D_1+1)(D_2+1)^2<(k N+1)^4$, one of the following statements holds:
    \begin{enumerate}
        \item there is $\mu=(\mu_0,\mu_1,\mu_2,\mu_3)$ with $\mu_i<k N$ such that $\log |F(z_\mu)| > - N^4\log N$;
        \item for all $\lambda=(\lambda_1,\lambda_2, \lambda_3)$ we have $\log |\psi(\lambda)|\le -\frac{1}{2} N^4\log N$.
    \end{enumerate}
Indeed, by Lemma~\ref{lem:tijdeman}, we have
\[
    A\leq (k N+1)^4\left(\frac{6 c_1 D_1}{a_0(D_2+1)}\cdot\frac{(D_1+1)(D_2+1)^2}{b}\right)^{(D_1+1)(D_2+1)^2}\left(\frac{72 c_2 k N}{b_0(k N+1)^2}\right)^{(k N+1)^4}\cdot E,
\]
where the quantities appearing in the formulas have already been introduced in Lemma~\ref{lem:tijdeman}.\\
Since we do not have much control over the distribution of the $z_\mu$'s, we bound $a_0, b_0\ge 1$. For $N$ sufficiently large, $\frac{72c_2 N}{b_0(k N+1)^2}<1$, so we obtain an estimate of the form
\[
\log A \ll (D_1+1)(D_2+1)^2 + \log E.
\]
If we assume that the first statement in the lemma does not hold, we infer that $\log E\le - N^4\log N$. As a result, 
\[
\log A \ll (D_1+1)(D_2+1)^2 - N^4\log N \ll (kN+1)^4 - N^4\log N \le - \frac{1}{2} N^4\log N,
\]
which is precisely the second statement above.

\smallskip

\emph{Case 1.}
 Let us suppose that there is $\mu=(\mu_0,\mu_1,\mu_2,\mu_3)$ with $\mu_i<k N$ such that $\log |F(z_\mu)| > - N^4\log N$.\\
 With the definition given above of $C_\mu$, it is clear that $C_\mu F(z_\mu)$ is a polynomial expression in the ring $\mathbb{Z}\left[\beta, \alpha_1,\alpha_2,\frac{\log \alpha_2}{\log\alpha_1},\alpha^{\beta}_1, \alpha^\beta_2\right]$. Let us estimate degree and height of the given expression of $C_\mu F(z_\mu)$. Notice that 
 \[
 F(z_\mu)=\sum_\lambda \psi(\lambda) \cdot C_\mu a_{\lambda\mu}.
 \]
 We have already seen that 
 \[
 \deg \psi(\lambda)\ll D_1+D_2 N, \qquad \log H(\psi(\lambda))\ll D_1\log N
 \]
 and that there are upper bounds of the same shape for the degree and height of $C_\mu a_{\lambda\mu}$. The fact that the latter estimates remain the same up to absolute constants depends on the assumption $|\mu_i|\le k N$, so $k$ can be included in an absolute constant. It follows immediately that 
 \[
    \deg C_\mu F(z_\mu) \ll D_1+D_2 N.
 \]
 For the height, we need to apply Gelfond's lemma \cite[Page 47, Lemma 3.3]{nesterenko} to the polynomials $\psi(\lambda)$ and $C_\mu a_{\lambda\mu}$ to get
 \[
 \log H(\psi(\lambda)C_\mu a_{\lambda\mu}) \ll D_1\log N +D_2 N.
 \]
 By counting the number of terms in the sum, we conclude that 
 \[
 \log H(C_\mu F(z_\mu)) \ll D_1\log N+D_2 N.
 \]
 Write for short $T=t^{\delta C_1(D_1+D_2N)}$, where $t$ is a denominator for $\alpha_1$, $\alpha_2$, $\beta$ over $\mathbb{Q}$. We then have
 \[
 TC_\mu F(z_\mu)\in \mathbb{Z}\left[ \beta, \alpha_1, \alpha_2, \frac{\log \alpha_2}{\log\alpha_1}, \alpha_1^{\beta}, \alpha_2^\beta \right].
 \]
 Let $A_1$ be the number obtained from the given polynomial expression of $TC_\mu F(z_\mu)$ by substituting $\alpha_1^\beta$ with $\xi_1$ and $\alpha_2^\beta$ with $\xi_2$, so that 
 \[
 A_1\in \mathbb{Z}\left[ \beta, \alpha_1, \alpha_2, \frac{\log \alpha_2}{\log\alpha_1}, \xi_1, \xi_2 \right].
 \]
 Formally, the polynomial expression is unchanged, thus
 \[
 \deg A_1 \ll D_1+D_2 N, \qquad \log H(A_1) \ll D_1\log  N+D_2 N.
 \]
 Let us write $A_1(x,y)$ for the polynomial obtained by replacing $\xi_1$ and $\xi_2$  in $A_1$ with the formal variables $x$ and $y$ respectively. In order to estimate $|A_1|$, we notice that 
 \begin{align*}
     |A_1(\xi_1,\xi_2)|\le |A_1(\alpha_1^\beta,\alpha_2^\beta)|+|A_1(\alpha_1^\beta,\alpha_2^\beta)-A_1(\alpha_1^\beta,\xi_2)| + |A_1(\alpha_1^\beta,\xi_2)-A_1(\xi_1,\xi_2)|.
 \end{align*}
 Let us go through these summands one by one. First,
 \begin{align*}
 \log |A_1(\alpha_1^\beta,\alpha_2^\beta)|= \log |TC_\mu F(z_\mu)| & \le \delta C_1(D_1+D_2 N)\log T + \log C_\mu -\frac{\varepsilon}{4} N^4\log N \\ &\le -\frac{\varepsilon}{5} N^4\log N.
 \end{align*}
 On the other hand, the second summand is formally divisible by $\alpha_2^\beta-\xi_2$. The modulus of the factor $\frac{A_1(\alpha_1^\beta,\alpha_2^\beta)-A_1(\alpha_1^\beta,\xi_2)}{\alpha_2^\beta-\xi_2}$ can be estimated by considering the degree and height of $A_1$. From this, we deduce
 \[
    \log |A_1(\alpha_1^\beta,\alpha_2^\beta)-A_1(\alpha_1^\beta,\xi_2)|\ll -r^{3C/4}+(D_1\log  N+D_2 N).
 \]
 A similar inequality holds for $\log |A_1(\alpha_1^\beta,\xi_2)-A_1(\xi_1,\xi_2)|$. It thus follows that 
 \[
 |A_1(\alpha_1^\beta,\alpha_2^\beta)-A_1(\xi_1,\xi_2)|\ll -r^{3C/4}+(D_1\log N+D_2 N).
 \]
 This means that $A_1(\xi_1,\xi_2)$ lies in a ball centered at $A_1(\alpha_1^\beta,\alpha_2^\beta)$ of radius at most $-r^{3C/4}+(D_1\log N+D_2 N)$ up to some absolute constant. If we now recall that 
 \[
 -N^4\log N \le \log |F(z_\mu)| \le -\frac{\varepsilon}{4} N^4\log N,
 \]
 we deduce that 
 \[
 -2 N^4\log N \le \log |A_1(\xi_1,\xi_2)| \le -\frac{\varepsilon}{5} N^4\log N.
 \]
 Notice in particular that $A_1(\xi_1,\xi_2)$ is not zero.\\
 Let us now set $A_2=A_2\left(\frac{\log \alpha_2}{\log \alpha_1},\xi_1,\xi_2\right)$ to be the norm of $A_1$ over $\mathbb{Q}\left(\frac{\log \alpha_2}{\log \alpha_1},\xi_1,\xi_2\right)$. Given the presence of the factor $T$ in $A_1$, it is clear that $A_2$ is a polynomial with integer coefficients in $\frac{\log \alpha_2}{\log \alpha_1},\xi_1,\xi_2$. Moreover, by Lemma~\ref{lemma: height for norm} it follows that 
 \[
 \deg A_2 \ll D_1+D_2 N, \qquad H(A_2)\ll D_1\log N + D_2 N.
 \]
 Let us denote by $\sigma_1,\dots, \sigma_\delta$ the Galois automorphisms of the Galois closure of the field $\mathbb{Q}\left(\beta,\alpha_1,\alpha_2,\frac{\log \alpha_2}{\log \alpha_1},\xi_1,\xi_2\right)$ over $\mathbb{Q}\left(\frac{\log \alpha_2}{\log \alpha_1},\xi_1,\xi_2\right)$. For each $i=1,\dots, \delta$, all the Galois conjugates $\sigma_i(A_1)=T\sigma_i(C_\mu F(z_\mu))$ are polynomials in $\mathbb{Z}\left[\sigma_i(\beta), \sigma_i(\alpha_1),\sigma_i(\alpha_2),\frac{\log \alpha_2}{\log\alpha_1},\alpha^{\beta}_1, \alpha^\beta_2\right]$ that share the same degree and height. Since the coefficients $\psi(\lambda)$ of $F$ are invariant under $\sigma_i$, we have $\sigma_i(F(z_\mu))=F(\sigma_i(z_\mu))$. Moreover, $z_\mu$ is a linear combination of $1,\sigma_i(\beta),\frac{\log \alpha_2}{\log \alpha_1}, \sigma_i(\beta)\frac{\log \alpha_2}{\log \alpha_1}$ with integer coefficients. Thus, the function $F\circ \sigma_i$ has the same number of zeroes as $F$ and these are contained in balls of comparable radii, up to absolute constants. It is then straightforward to see that Schwarz's lemma can be applied to $F\circ \sigma_i$ to get the same estimates that have already been established for $F$. As a result, it follows that 
 \[
    \log |A_2|\le -\frac{\varepsilon}{5}\delta N^4\log N \le -\frac{\varepsilon}{5}N^4\log N.
 \]
 If we consider $A_2$ as a polynomial in $\xi_2$, then either its leading coefficient $w\left(\frac{\log \alpha_2}{\log \alpha_1}, \xi_1\right)$ or $A_2/w$ has modulus at most
 \[
    \exp\left( - \frac{\varepsilon}{10} N^4\log N\right).
 \]
 If $|w|$ satisfies this inequality, we let $A_3=w$. Otherwise, we define $A_3$ to be the semi-resultant of $A_2\left(\frac{\log \alpha_2}{\log \alpha_1}, \xi_1,z\right)$ and $Q\left(\frac{\log \alpha_2}{\log \alpha_1}, \xi_1,z\right)$ with respect to the variable $z$. By Lemma~\ref{lemma: degree and height of semi-resultant}, it follows that 
 \begin{align*}
    \deg_{\frac{\log \alpha_2}{\log \alpha_1}} A_3 & \ll (D_1+D_2 N)(\deg_x Q + \deg_z Q),\\
    \deg_{\xi_1} A_3 & \ll (D_1+D_2 N)(\deg_y Q + \deg_z Q),\\
    \log \text{ht}\, A_3 & \ll (D_1+D_2 N)(\log H(Q) +\deg Q)+(D_1\log N+D_2 N)\deg_z Q,\\
    \log |A_3| & \le -\frac{\varepsilon}{10} N^4\log N+2(D_1+D_2 N)(\deg Q+\log H(Q)) +(D_1\log N+D_2 N)\deg_z Q.
 \end{align*}
 Since
 \[
    N^4\log N > N+2(D_1+D_2 N)(\deg Q +\log \text{ht}\, Q)+(D_1\log N+D_2 N)\deg_z Q,
 \]
 the negative term in the upper bound for $\log |A_3|$ dominates the other ones.
 Thus, we have 
\[
    \log |A_3| \le -\frac{\varepsilon}{11} N^4\log N.
\]
Notice that these inequalities are automatically satisfied by $A_3$ also in the case when $A_3=w$.\\
Now, let $A_4$ be the semi-resultant of $A_3\left(\frac{\log \alpha_2}{\log \alpha_1}, y\right)$ and $P\left(\frac{\log \alpha_2}{\log \alpha_1}, y\right)$ with respect to $y$. We then have
\begin{align*}
    \deg_{\frac{\log \alpha_2}{\log \alpha_1}} A_4  \ll &\deg_y P(D_1+D_2 N)(\deg_x Q + \deg_z Q) \\ &+\deg_x P (D_1+D_2 N)(\deg_y Q + \deg_z Q),\\
    \log H(A_4)  \ll & (D_1+D_2 N)((\deg_yQ+\deg_zQ)(\deg_yP+\deg_xP)\\& +\deg_yP(\deg_xQ+\deg_z Q+\log H( Q ) + \deg Q)\\
    &+\log H( P)(\deg_y Q+\deg_z Q))+(D_1\log N +D_2 N)\deg_yP\deg_z Q,\\
    \log |A_4|\le &  -\frac{\varepsilon}{11}N^4\log N+2(D_1+D_2 N)(\deg_y Q+\deg_zQ+\log H( P)\\
    &+\deg_y P(\deg Q+ \log H(Q))) +(D_1\log N+D_2 N)\deg_z Q.
 \end{align*}
 Thus, setting 
 \begin{align*}
     A=&\deg_yP (\deg_x Q+ \deg_z Q) + \deg_x P (\deg_y Q+ \deg_z Q),\\
     B=&(\deg_yQ+\deg_zQ)(\deg_yP+\deg_xP)\\
     & +\deg_yP(\deg_xQ+\deg_z Q+\log H( Q) + \deg Q)\\
     &+\log H( P)(\deg_y Q+\deg_z Q),
 \end{align*}
 we conclude that 
 \begin{align*}
     \deg_{\frac{\log \alpha_2}{\log \alpha_1}} A_4  \ll &A(D_1+D_2 N),\\
     \log H( A_4)  \ll & B(D_1+D_2 N) +\deg_yP\deg_z Q(D_1\log N +D_2 N).
 \end{align*}
 Moreover, given that
 \begin{align*}
    N^4\log N>& 2(D_1+D_2 N)(\deg_y Q+\deg_zQ+\log H(P)\\
    & +\deg_y P(\deg Q+ \log H( Q))) +(D_1\log N+D_2 N)\deg_z Q
 \end{align*}
 we also get 
 \begin{equation}\label{polynomial}
    \log |A_4| \le -\frac{\varepsilon}{12} N^4\log N.
 \end{equation}
 \emph{Case 2.}\\
 Let us assume that for all $\lambda=(\lambda_1,\lambda_2,\lambda_3)$ we have 
 \[
    \log |\psi(\lambda)|\le -\frac{1}{2}N^4\log N.
 \]
 Since $\psi(\lambda)\in \mathbb{Z}\left[\frac{\log\alpha_2}{\log \alpha_1}, \alpha_1^\beta,\alpha_2^\beta \right]$, we can write $\psi(\lambda)=\Psi_\lambda\left(\frac{\log\alpha_2}{\log \alpha_1}, \alpha_1^\beta,\alpha_2^\beta\right)$ for some polynomial $\Psi_\lambda\in\mathbb{Z}[x,y,z]$. The fact that the $\Psi_\lambda$'s do not have a common factor implies that the polynomials $\Psi_\lambda\left(\frac{\log\alpha_2}{\log \alpha_1}, \xi_1,z\right)$ are not all zero. If there is a non-zero $\Psi_\lambda\left(\frac{\log\alpha_2}{\log \alpha_1}, \xi_1,z\right)$ whose leading coefficient has absolute value at most 
 \[
    \exp\left( -\frac{\varepsilon}{10} N^4\log N\right),
 \]
 this leading term satisfies the same inequalities as $A_3$ as far as degrees, height and absolute value are concerned. \\
 Otherwise, let us assume that all these leading terms have absolute value at least 
 \[\exp\left( -\frac{\varepsilon}{10}N^4\log N\right).\]
 In this case, let $\Psi'(\lambda)$ be the monic polynomial obtained from $\Psi(\lambda)$ by dividing by its leading coefficient. If $Q_1$ is the minimal monic polynomial of $\xi_2$ over $\mathbb{Q}\left(\frac{\log\alpha_2}{\log \alpha_1}, \xi_1\right)$, consider the maximal non-negative integer $\sigma$ such that 
 \[
    \Psi_\lambda\left(\frac{\log\alpha_2}{\log \alpha_1}, \xi_1,z\right)=Q_1(z)^\sigma R_\lambda(z)
 \]
 for some polynomial $R_\lambda$ in $\mathbb{Q}\left(\frac{\log\alpha_2}{\log \alpha_1}, \xi_1\right)[z]$. \\
 If $\sigma=0$, there is some $\Psi_\lambda\left(\frac{\log\alpha_2}{\log \alpha_1}, \xi_1,z\right)$ which is not divisible by $Q_1$. As a result, the modulus of the semi-resultant of this $\Psi_\lambda\left(\frac{\log\alpha_2}{\log \alpha_1}, \xi_1,z\right)$ together with $Q\left(\frac{\log\alpha_2}{\log \alpha_1}, \xi_1,z\right)$ can be estimated by means of Lemma~\ref{lemma: chudnovsky} by evaluating $\Psi_\lambda'\left(\frac{\log\alpha_2}{\log \alpha_1}, \xi_1,z\right)$ and $Q_1(z)$ at $\xi_2$. This procedure yields the same estimates for this semi-resultant as the ones obtained for $A_3$.\\
 Suppose on the other hand that $\sigma>0$. This prevents us from applying Lemma~\ref{lemma: chudnovsky} to estimate the modulus of the semi-resultant, since $\Psi_\lambda\left(\frac{\log\alpha_2}{\log \alpha_1}, \xi_1,z\right)$ and $Q\left(\frac{\log\alpha_2}{\log \alpha_1}, \xi_1,z\right)$ are never coprime. Let us therefore replace the auxiliary function $F$ with the function $G$ obtained from $F$ by replacing each $\psi(\lambda)$ with $R_\lambda(\alpha_2^\beta)$. Essentially, this has the effect of substituting each occurrence of $\alpha_1^\beta$ in $\psi(\lambda)$ with $\xi_1$ and subsequently dividing by $Q_1(\alpha_2^\beta)^\sigma$.\\
 For all $\mu=(\mu_0,\dots, \mu_3)$ with $|\mu_i|\le N\log N$ we have
 \[
    t^{\delta C_1(D_1+D_2N)}G(z_\mu) = S_\mu(\alpha_1^\beta)=\sum_{j}\chi_j \alpha_1^{\beta j}
 \]
 for some polynomial $S_\mu$ in $\alpha_1^\beta$ with coefficients $\chi_j\in \mathbb{Z}\left[\beta, \alpha_1,\alpha_2,\frac{\log\alpha_2}{\log \alpha_1}, \xi_1,\alpha_2^\beta\right]$. It is clear that $\deg S_\mu \ll D_1+D_2 N$. Moreover, we observe that each $\chi_j$ has degree $\ll D_1+D_2 N$ and logarithm of height $\ll D_1\log N+D_2 N$. To prove this assertion, we need to bound the height of $R_\lambda(z)$. By \cite[Lemma 3]{tijdeman1} we have
 \[
    H(R_\lambda(z)) \le H(Q_1(z))^{-1} H\left(\Psi_\lambda\left(\frac{\log\alpha_2}{\log \alpha_1}, \xi_1,z\right)\right) \exp\left(\deg\Psi_\lambda \left(\frac{\log\alpha_2}{\log \alpha_1}, \xi_1,z\right) \right).
 \]
 Since $Q_1$ is monic, we have $H(Q_1(z))\ge 1$. From the upper bounds at our disposal for the degree and height of $\Psi_\lambda$ it follows that $\log H(R_\lambda(z))\ll D_1\log N+D_2 N$. This estimate then carries on to $\log H(\chi_j)$ by the shape of $\chi_j$.\\
 Recall that the function $F$ has been constructed by finding the polynomials $\Psi_\lambda(x,y,z)$ as the solution of a certain linear system in which we have formally replaced $\frac{\log\alpha_2}{\log \alpha_1}, \alpha_1^\beta,\alpha_2^\beta$ by $x,y,z$. This means that the expression obtained from $F(z_\mu)$ by replacing every occurrence of $\alpha_1^\beta$ with $\xi_1$ still vanishes. Dividing said expression by $Q(\alpha_2^\beta)^\sigma$ coincides with $S_\mu(\xi_1)$, which is therefore identically zero. As a result
 \[
    |S_\mu(\alpha_1^\beta)| =|S_\mu(\xi_1)+S_\mu(\alpha_1^\beta)-S_\mu(\xi_1)| = |S_\mu(\alpha_1^\beta)-S_\mu(\xi_1)| \le \sum_j |\chi_j||\alpha_1^{\beta j}-\xi_1^j|.
 \]
 Estimates for degree and height of $\chi_j$ readily imply that $|\chi_j|\ll D_1\log N+D_2N$. Since $|\alpha_1^{\beta j}-\xi_1^j|$ is formally divisible by $|\alpha_1^\beta-\xi_1|$, we infer that
 \[
    \log |S_\mu(\alpha_1^\beta)| \le -\frac{1}{2}r^{7C/8}.
 \]
 Now we  want to estimate the modulus of $G$ on the ball of radius $N^{1+\varepsilon}$.
 So far, we have seen that $G$ has a small value at all points $z_\mu$ with $|\mu_i|\le N$. Since it need not vanish at these points, instead of applying Schwarz's lemma we use Hermite's interpolation (see for the reference \cite[Theorem 4.2]{lang}):
 \[
    G(z)=\frac{1}{2\pi i}\int_\Gamma \frac{G(t)}{t-z}\prod_{\mu}\frac{z-z_\mu}{t-z_\mu}dt - \frac{1}{2\pi i}\sum_\mu G(z_\mu)\int_{\gamma_\mu} \prod_{\mu'\neq \mu}\frac{z-z_\mu}{t-z_{\mu'}}dt
 \]
where $\Gamma$ is the circle of radius $N^2$ centered at the origin and $\gamma_\mu$ is a circle of sufficiently small radius centered at $z_\mu$. \\
 In order to estimate $|G(z)|$ for $|z|=N^{1+\varepsilon}$, notice that
 \begin{itemize}
     \item[(a)] $|z-z_\mu|\le 2 N^{1+\epsilon}$;
     \item[(b)] $|t-z_\mu|\ge N^2-N\ge \frac{1}{2}N^2$;
     \item[(c)] $|t-z|\ge N^2-N^{1+\epsilon}\ge \frac{1}{2}N^2$.
 \end{itemize}
     Overall:
     \[
        \left| \frac{1}{t-z}\prod_{\mu}\frac{z-z_\mu}{t-z_\mu}\right|\le \frac{2}{N^2}\prod_\mu \frac{4N^{1+\epsilon}}{N^2}\le \exp \left(-\varepsilon N^4\log N \right) 
     \]
     $|G(t)|$ can be estimated as we did for $F$ before apply Schwarz's lemma. This is smaller than $D_1\log N + D_2N^{1+\varepsilon}$ up to absolute constants. Moreover 
     \[
        \log\left| \prod_{\mu'\neq \mu}\frac{z-z_\mu}{t-z_\mu'}\right|\le -\varepsilon N^4\log N.
     \]
 Putting this together, we get
 \[
    |G(z)|\le \max\left\{ -\frac{\varepsilon}{16}N^4\log N, -r^{7C/8} \right\}.
 \]
 Since
 \[
    N^4\log N<r^{3C/4}<r^{7C/8},
 \]
 the term $-N^4\log N$ dominant.
 Thus, when $|z|<N^{1+\varepsilon}$ we have 
 \[
    \log|G(z)|<-\frac{\varepsilon}{17}N^4\log N,
 \]
 which holds, in particular when $z=z_\mu$.\\
 We now apply Lemma~\ref{lem:tijdeman} to the function $G(z)$. Since the estimate of degree, height and modulus for $G$ and its coefficients mirror the ones of $F$, either there is $\mu=(\mu_0,\mu_1,\mu_2,\mu_3)$ with $|\mu_i|\le R$ such that $\log|G(z_\mu)|\gg -N^4\log N$ or otherwise  $\log|R_\lambda|\le -\frac{1}{2}N^4\log N$ for all $\lambda$.\\
 This time, the latter case is easier to handle. By the definition of $\sigma$, there is some $R_\lambda(z)$ which is relatively prime to $Q_1(z)$. Taking their semi-resultant with respect to $z$, we obtain a polynomial $B_3\in\mathbb{Z}\left[\frac{\log\alpha_1}{\log \alpha_2}, \xi_1\right]$ which satisfies 
 \begin{align*}
    \deg_{\frac{\log \alpha_2}{\log \alpha_1}} B_3 & \ll (D_1+D_2N)(\deg_x Q + \deg_z Q),\\
    \deg_{\xi_1} B_3 & \ll (D_1+D_2N)(\deg_y Q + \deg_z Q),\\
    \log H( B_3) & \ll (D_1+D_2N)(\log H(Q) +\deg Q)+(D_1\log N+D_2N)\deg_z Q,\\
    \log |B_3| & \le -\frac{\varepsilon}{11}N^4\log N.
 \end{align*}
 To estimate $|B_3|$, one can apply Lemma~\ref{lemma: chudnovsky} to the polynomials $Q_1(z)$ and $R_\lambda(z)$ divided by their leading coefficients, evaluated at $\alpha_2^\beta$. Notice that the leading coefficient of $R_\lambda(z)$ is the same as $\Psi_\lambda\left(\frac{\log \alpha_2}{\log \alpha_1},\xi_1,z\right)$, as $Q_1(z)$ is monic, and the logarithm of its absolute value was assumed to be $>-\frac{\varepsilon}{10}N^4\log N$. Also, $Q_1(\xi_2)=0$, so $Q_1(\alpha_2^\beta)=Q_1(\alpha_2^\beta)-Q_1(\xi_2)$, which is formally divisible by $\alpha_2^\beta-\xi_2$. Usual estimates give $\log|Q_1(\alpha_2^\beta)|<-\frac{1}{2}r^{7C/8}$. 
 \smallskip
 
 For the other case, let $A_1'$ be the polynomial obtained from $TC_\mu F(z_\mu)$ by replacing $\alpha_1^\beta$ with $\xi_1$ and $\alpha_2^\beta$ with a formal variable $z$. We thus have 
 \[
    A_1'\in \mathbb{Z}\left[\alpha_1,\alpha_2,\beta, \frac{\log\alpha_2}{\log\alpha_1},\xi_1,z\right].
 \]
 Let $A_2'$ be the norm of $A_1'$ with respect to $\mathbb{Q}\left(\frac{\log\alpha_2}{\log\alpha_1},\xi_1,z\right)$, so
 \[
    A_2'\in  \mathbb{Z}\left[\frac{\log\alpha_2}{\log\alpha_1},\xi_1,z\right].
 \]
 We observe that $A_2'\neq 0$. To prove this, it is enough to show that $A_1'$ is not zero. Recall that $p\left(\frac{\log\alpha_2}{\log\alpha_1}\right)$ is a denominator for $\xi_1$, while $q\left(\frac{\log\alpha_2}{\log\alpha_1},\xi_1\right)$ is a denominator for $\xi_2$. Let $B_1$ be the polynomial obtained from
 \[
    TC_\mu\left(\frac{\log\alpha_2}{\log\alpha_1}\right)^{c(D_1+D_2N)}q\left(\frac{\log\alpha_2}{\log\alpha_1},\xi_1\right)^{c'(D_1+D_2N)}G(z_\mu)
 \]
 for some suitable constants $c,c'>0$ by replacing every occurrence of $\alpha_1^\beta$ with $\xi_1$ and of $\alpha_2^\beta$ with $z$. Then $B_1$ is obtained form $A_1'$ by multiplying by some non-zero numbers and dividing by $Q_1(z)^\sigma$. Thus, $A_1'$ vanishes if and only if $B_1=0$. However, the lower bound for $|G(z_\mu)|$ implies that $B_1\neq 0$ in the same way as we did for $A_1$. It is also immediate to see that one has estimates for the degree and height of $A_2'$ analogous to those of the polynomial $A_2$ constructed previously.\\
 Now let $A_3'$ be the semi-resultant of $A_2'\left(\frac{\log\alpha_2}{\log\alpha_1}, \xi_1, z\right)$ and $Q\left(\frac{\log\alpha_2}{\log\alpha_1}, \xi_1,z\right)$ with respect to the variable $z$. Degree and height of $A_3'$ satisfy inequalities similar to those of $A_3$, up to absolute constants. In order to estimate $|A_3'|$, we argue as follows.
 Notice that $Q\left(\frac{\log\alpha_2}{\log\alpha_1},\xi_1,\xi_2\right)=0$ by definition of $\xi_2$, so the irreducible monic polynomial $Q_1\left(\frac{\log\alpha_2}{\log\alpha_1},\xi_1,z\right)$ divides $Q\left(\frac{\log\alpha_2}{\log\alpha_1},\xi_1,z\right)$. On the other hand, let $w\left(\frac{\log\alpha_2}{\log\alpha_1},\xi_1\right)$ be the leading coefficient of $A_2'\left(\frac{\log\alpha_2}{\log\alpha_1},\xi_1,z\right)$. Suppose that  
 \[
    \log|w|\ge -\frac{\varepsilon}{10}N^4\log N.
 \]
  Then we apply Lemma~\ref{lemma: chudnovsky} to $Q_1\left(\frac{\log\alpha_2}{\log\alpha_1},\xi_1,z\right) $ and $A_2'\left(\frac{\log\alpha_2}{\log\alpha_1},\xi_1,z\right)/w\left(\frac{\log\alpha_2}{\log\alpha_1},\xi_1\right)$ evaluated at $z=\xi_2$. The evaluation of $A_2'$ at $z=\xi_2$ coincides with $A_2$, for whose modulus we had already found an upper bound. It then follows that 
  \[
    \log|A_3'|\le -\frac{\varepsilon}{11}N^4\log N.
  \]
  Otherwise, we have 
  \[
    \log\left|w\left(\frac{\log\alpha_2}{\log\alpha_1},\xi_1\right)\right|\le -\frac{\varepsilon}{10}N^4\log N \le -\frac{\varepsilon}{11}N^4\log N.
  \]
  Estimates for degree and height of $w$ follow immediately from the ones of $A_2'$.\\
  Thus, in any case we have found a polynomial in $\mathbb{Z}\left[\frac{\log\alpha_2}{\log\alpha_1}, \xi_1\right]$ which has degree, height and modulus bounded in the same way as $A_3$, whether it be $w$, $A_3'$ or $B_3$. By taking the semi-resultant with $P\left(\frac{\log\alpha_2}{\log\alpha_1},y\right)$ with respect to $y$ we finally obtain a polynomial in $\mathbb{Z}\left[\frac{\log\alpha_2}{\log\alpha_1}\right]$ which satisfies the same inequalities as $A_4$ as far as degree, height and modulus are concerned.

\smallskip 
\noindent{\it Step 2.} 
 In the previous section, we have constructed a polynomial $s\in\mathbb{Z}\left[\frac{\log\alpha_2}{\log\alpha_1}\right]$, which we had called $A_4$, such that
 \begin{align*}
     \deg_{\frac{\log \alpha_2}{\log \alpha_1}} s  \ll &A(D_1+D_2N),\\
     \log H( s)  \ll & B(D_1+D_2N) +\deg_yP\deg_z Q(D_1\log N +D_2N),\\
      \log |s| \le &-\frac{\varepsilon}{12}N^4\log N.
 \end{align*}
 Notice that this polynomial depends on the parameter $N$, so we will denote it by $s_N$ to make this dependence visible. \\
 By Lemma \ref{irreducible}, there is a factor $t_N$ of $s_N$ which is a power of an irreducible polynomial and satisfies
 \begin{align*}
     \deg_{\frac{\log \alpha_2}{\log \alpha_1}} t_N  \ll &A(D_1+D_2N),\\
     \log H( t_N) \ll & (A+B)(D_1+D_2N) +\deg_yP\deg_z Q(D_1\log N +D_2N),\\
      \log |t_N| \le &-\frac{\varepsilon}{13}N^4\log N.
 \end{align*}
 Notice that, in order to apply Lemma~\ref{irreducible}, we need to check that 
 \[
    \frac{\varepsilon}{12}N^4\log N > \deg s (\log H(s)+\deg s)
 \]
 asymptotically for $N$ sufficiently large. This is implied by the fact that 
 \[
 \frac{\varepsilon}{13}(N-1)^4\log (N-1) > c A\deg_y P\deg_z Q (D_1\log N)(D_1+D_2N)
 \]
 for a sufficiently large constant $c>0$. Moreover, in order to estimate the height of $t_N$ we use Lemma 2 of \cite{bro1}.\\
 Let us write $t_N=u_N^{v_N}$ for some $v_N\in \mathbb{N}$ and $u_N\in \mathbb{Z}\left[\frac{\log \alpha_2}{\log \alpha_1}\right]$ irreducible. Consider the resultant $r_{N}$ of $t_{N-1}$ and $t_{N}$, which is an integer. Using Lemma~\ref{lemma: chudnovsky}, we can estimate $|r_N|$ by evaluating $t_{N-1}$ and $t_{N}$ at $\frac{\log \alpha_2}{\log \alpha_1}$. We then have 
 \begin{align*}
    \log |r_N| \le & -\frac{\varepsilon}{13}(N-1)^4\log (N-1) \\
               & +c_1 A^2 (D_1+D_2N)^2 +c_1A(A+B)(D_1+D_2N)^2 \\
               & +c_1 A\deg_yP\deg_z Q (D_1\log N+D_2N)(D_1+D_2N) \\
               \le & -\frac{\varepsilon}{13}(N-1)^4\log (N-1)\\
               & + c A\deg_yP\deg_z Q (D_1\log N)(D_1+D_2N)
 \end{align*}
 Given that the inequality
 \[
    \frac{\varepsilon}{13}(N-1)^4\log (N-1) > cA\deg_yP\deg_z Q (D_1\log N)(D_1+D_2N)
 \]
 holds, it follows that for large $N$ we have $|r_N|<1$, so $r_N=0$ given that $r_N\in \mathbb{Z}$. Since $t_N$ and $t_{N+1}$ are powers of precisely one irreducible polynomial, the fact that their resultant is zero implies that $u_N=u_{N+1}$. We will therefore denote by $u$ the polynomial $u_N$ for all $N$.
\smallskip

 \noindent{\it Step 3.}
 Recall that the free parameter $N$ has been chosen such that $N_0\le N\le N_1$, with $N_0$ and $N_1$ defined at the beginning of the proof. By the estimates in the previous section, we have
 \[
    \log \left|u\left(\frac{\log\alpha_2}{\log\alpha_1}\right)\right| \le -\frac{\varepsilon}{13}\frac{N^4\log N}{v_N}\ll -\frac{N^4\log N}{A(D_1+D_2N)}.
 \]
 By choosing $N=N_1$, one obtains precisely the upper bound for $|u|$ in the statement of Theorem~\ref{theorem: main}. On the other hand, applying Lemma 2 in \cite{bro1} for the height of a factor, we have
 \[
    \deg  u + \log H(u) \ll (2A+B)(D_1+D_2N)+ \deg_yP\deg_z Q(D_1\log N +D_2N).
 \]
 By choosing $N=N_0$, we deduce the estimate in the statement of the Theorem~\ref{theorem: main} for degree and height of $u$. This concludes the proof of Theorem~\ref{theorem: main}.

\section{Proof of Theorem \ref{theorem:main1}}
    In order to prove Theorem~\ref{theorem:main1}, we need the following result of Waldschmidt \cite[Corollary 4.5, Theorem 4.7]{waldschmidt}.
    \begin{theorem}\label{waldschmidt}
          Let $\alpha_1$, $\alpha_2$ be algebraic numbers such that $\frac{\log \alpha_2}{\log\alpha_1}$ is irrational and $\beta$ be  quadratic irrational. Let $P(X)\in\mathbb{Z}[X]$ be a non-constant polynomial of degree $N$ and height at most $H$. Then 
          there exist positive constant $C_2(\log\alpha_i,\beta)$ and $C_3(\alpha_1,\alpha_2)$ such that 
          \begin{equation*}
                \log|P(\alpha_i^\beta)|\geq \frac{-C_2(\log\alpha_i, \beta)N^3(\log H+\log N)(\log \log H+\log N)}{(1+\log N)^2},  
          \end{equation*}
          and
          \begin{equation*}
                \log\left|P\left(\frac{\log\alpha_2}{\log \alpha_1}\right)\right|\geq -2 C_3(\alpha_1,\alpha_2)\frac{N^3 (\log H+N\log N)}{(1+\log N)^2}.
          \end{equation*}
    \end{theorem}
    
\noindent{\it Proof of Theorem \ref{theorem:main1}}.
Suppose that the statement of Theorem~\ref{theorem:main1} is not satisfied. Fix some $C_1>0$, which we will choose later to be sufficiently large. Then for all $C>C_1$ there are two relatively prime polynomials $P(x,y)$ and $Q(x,y,z)$ such that 
$$
    \log\mbox{max}\left\{\left|P\left(\frac{\log \alpha_2}{\log \alpha_1}, \alpha_1^\beta\right)\right|,\left|Q\left(\frac{\log \alpha_2}{\log \alpha_1}, \alpha_1^\beta, \alpha^\beta_2\right)\right| \right\}\leq -r^{C}.
$$
Then by Theorem \ref{theorem: main}, there is a choice of $C>C_1$ for which we obtain a polynomial $U(x)\in\mathbb{Z}[x]$ such that 
$$
    \log\left|U\left(\frac{\log\alpha_2}{\log\alpha_1}\right)\right|<-r^{C/4}
$$ 
 and, if we let $t$ be the type of $U$, then
 \[
 t\le -r^{C/C_1}.
 \]
 This means that 
 \[
 \log \left|U\left(\frac{\log\alpha_2}{\log\alpha_1}\right)\right|\le -t^{\frac{C_1}{4}}.
 \]
 On the other hand, setting $N=\deg U$ and $H=H(U)$, by Theorem \ref{waldschmidt}
 \[
    \log|U(\alpha^\beta)|>-2 C_3\frac{N^3 (\log H+N\log N)}{(1+\log N)^2}.
 \]
 If we now choose $C_1>0$ sufficiently large, these two estimates lead to a contradiction.\\
 In order to obtain the version of the statement of Theorem~\ref{theorem:main1} with $P$ involving only $\alpha_1^\beta$ and $\alpha_2^\beta$, one can argue exactly as above, using Theorem~\ref{theorem: main} to produce a polynomial $U$ in $\alpha_1^\beta$ such that $\log|U(\alpha_1^\beta)|<-t(U)^{C_1/4}$. A contradiction then follows from the transcendence measure of $\alpha_1^\beta$ provided by Theorem~\ref{waldschmidt}.

\medskip 

\noindent
\textbf{Acknowledgements.} We are very grateful to Prof. Michel Waldschmidt for his valuable suggestions and encouragement. We also thank Prof. Johannes Sprang for giving us the opportunity to meet at the University of Duisburg-Essen during the first author's visit and for his interest and support in the project. The first author is thankful to NBHM-Research grant.  The second author acknowledges financial support from the DFG Research Training Group 2553.

\end{document}